\documentclass[11pt,a4paper]{article}
\usepackage{amsmath,amsfonts,amstext, amsthm, amssymb, tikz,  cite, colonequals, array, rotating,graphicx,ytableau}
\usetikzlibrary{patterns}
\usepackage{pgfplots}
\author{Mikl\'{o}s B\'{o}na and Marie-Louise Bruner}
\title{Log-concavity, the Ulam
distance and involutions\thanks{The second author was supported by the Austrian Science Foundation FWF, grant P25337-N23.}}
\date{\today}

\theoremstyle{plain}
\newtheorem{theorem} {Theorem} [section]

\newtheorem{conjecture} [theorem] {Conjecture}
\newtheorem{proposition} [theorem] {Proposition}
\newtheorem{definition} [theorem]{Definition}
\newtheorem{corollary} [theorem] {Corollary}

\theoremstyle{definition}

\newtheorem*{example*}{Example}
\theoremstyle{remark}

\usepackage{epstopdf}

\newcolumntype{P}[1]{>{\centering\arraybackslash}p{#1}}

\newcommand{\ra}{\rightarrow}
\newcommand{\da}{\downarrow}
\newcommand{\dashmapsto}{\mapstochar\dashrightarrow}
\newcommand{\dashdownarrow}{\raisebox{2.0ex}{\rotatebox{-90}{$\dashrightarrow$}}}
\newcommand{\dashuparrow}{\raisebox{-1.5ex}{\rotatebox{90}{$\dashrightarrow$}}}

\begin{document}

\maketitle
 
\begin{abstract} We prove that in a large collection of naturally defined sets of permutations of fixed length, the numbers of permutations
at Ulam distance $k$ from the identity  form a log-concave sequence in $k$.
\end{abstract}

\section{Introduction and background}

Let $a_0,a_1,\cdots ,a_n$ be a sequence of positive real numbers. We say that the sequence is {\em log-concave}
if for all indices $k$, the inequality $a_{k-1}a_{k+1} \leq a_k^2$ holds. Log-concave sequences play an important role
in Combinatorics; see Chapter 7 of \cite{handbook} for a recent survey by Petter Br\"and\'en on the subject. 

In the last ten years, there was significant interest in biologically motivated sorting algorithms and notions of distance
for permutations. A collection of these results can be found in \cite{genome}.  The second crucial notion of this paper, that of Ulam distances, is a special case of these. 
It was introduced by Ulam as an ``evolutionary distance'' in the context of biological sequences~\cite{ulam1972some}.

\begin{definition}
Given two permutations $\sigma$ and $\tau$ of the same length, the \emph{Ulam distance} $U(\sigma, \tau)$ is the minimal number of steps needed to obtain $\tau$ from $\sigma$ where each step consists in taking an element from the current permutation and placing it at some other position.
\end{definition}

If $\tau$ is the identity permutation $id$ of length $n$, then it is easy to see that $n- U(\sigma, \tau)$ is equal  to the length $\ell(\sigma)$ of the longest increasing subsequence in $\sigma$. Indeed, for any permutation $\sigma$ that is not the
identity permutation, we can find an allowed move that increases $\ell(\sigma)$ by one, but we can never find an allowed 
move that would increase $\ell(\sigma)$ by two. 

In what follows,  we will denote by $U_{n,k}$ the set of all permutations of length $n$ that have Ulam distance $n-k$ to the identity permutation. Equivalently, these are the permutations $\sigma$ that satisfy $\ell(\sigma)=k$. 
The cardinality of $U_{n,k}$ will be denoted by $u_{n,k}$. 

The algorithmic question of determining the length of the longest increasing subsequence can be answered in $O(n \log(n))$-time for sequences in general \cite{schensted1987longest} and in $O(n \log(\log(n)))$-time for permutations of length $n$~\cite{DBLP:journals/ipl/ChangW92}.
The distribution of the parameter $\ell(\sigma)$ has been
the subject of vigorous study for over 60 years. See \cite{baik} for strongest results on this subject, and see \cite{aldous} 
for a history of the problem. 
However, it is still not known whether the sequence $u_{n,1},u_{n,2},\cdots ,u_{n,n}$ is log-concave for each fixed $n$.

In this paper, we state our conjecture that the sequence mentioned in the previous sentence is indeed log-concave.
This conjecture is supported by numerical evidence and fits in a long line of facts \cite{bona-flynn} concerning other biologically
motivated sorting algorithms. 
Then we proceed to prove the conjecture in some special cases, that is, for certain subsets of permutations of length $n$,
as opposed to the entire set of $n!$ permutations of length $n$. 
One tool in our proofs will be a technique that allows us to turn injections between sets of involutions into injections between sets of permutations. In addition, we will use several consequences of the well-known Robinson-Schensted correspondence.
Recall that the Robinson-Schensted correspondence is a bijection that maps each permutation $p$ of length $n$ into an 
ordered pair $(P(p),Q(p))$ of  Standard Young Tableaux (SYT) of the same shape and of size $n$. 
In these two Standard Young Tableaux, the length of the first row corresponds to the length of the longest increasing subsequence in $p$.
Also recall the following: If the Robinson-Schensted correspondence maps $p$ into $(P(p),Q(p))$, then it maps the inverse permutation $p^{-1}$  to the pair $(Q(p),P(p))$.
Thus, if $p$ is an involution,  it corresponds to a pair $(P,P)$ and can be identified with the {\em single} SYT $P$.
See Chapter 14  of \cite{handbook} for a recent survey by Ron Adin and Yuval Roichman on Standard Young Tableaux.

\section{The conjecture and a first result}

Supported by the data that we computed for permutations of length up to $n=15$, we conjecture the following:

\begin{conjecture}
For every positive integer $n$ the sequence $u_{n,k}$ where $1 \leq k \leq n$ is log-concave.
\label{conj}
\end{conjecture}

Let $I_{n,k}$ denote the set of all involutions of length $n$ that have Ulam distance $n-k$ to the identity permutation respectively the set of all involutions of length $n$ with longest increasing subsequence of length $k$. 
The cardinality of $I_{n,k}$ will be denoted by $i_{n,k}$.

\begin{theorem} \label{invtoperm}
For every positive integer $n$ the following holds: If the sequence $(i_{n,k})_{1\leq k \leq n}$ is log-concave, then so is the sequence $(u_{n,k})_{1\leq k \leq n}$.
\label{thm:involutions_to_arbirary_permutations}
\end{theorem}

\begin{proof}
In order to show that the sequence $(u_{n,k})$ is log-concave it would suffice to find an injection from $U_{n,k-1} \times U_{n,k+1}$ to $U_{n,k} \times U_{n,k}$ for all $n \geq 1$ and $2 \leq k \leq n-1$.

Assume that the statement of log-concavity is true for involutions. 
Then there is an injection $f_{n,k}$ from $I_{n,k-1} \times I_{n,k+1}$ to $I_{n,k} \times I_{n,k}$ for all $n \geq 1$ and $2 \leq k \leq n-1$.
Now let $p_1\in U_{n,k-1}$ and $p_2\in U_{n,k-2}$. 
Then, they correspond to the pairs $(P_1,Q_1)$ and $(P_2,Q_2)$ of Standard Young Tableaux.
Define $F(p_1,p_2) = (w_1,w_2)$, where $w_1$ is the permutation in $U_{n,k}$ whose pair of SYT is $f(P_1,P_2)$, and $w_2$ is the permutation in $U_{n,k}$ whose pair of SYT is $f(Q_1,Q_2)$.
Then $F$ is injective since $f$ is injective. 
\end{proof}

In the following, we will apply this result to specific classes of permutations and show that our conjecture indeed holds there.

\section{A class of permutations for which the conjecture holds}

First we show that the conjecture holds for permutations whose corresponding SYT are \textit{hooks}.
This implies that the conjecture is true for the class of \textit{skew-merged involutions}.
Then we will define a generalization of hooks, introducing $(l,m)$-protected SYT.
We will see that the conjecture holds for these much larger classes of SYT respectively permutations as well, for every pair $(l,m)$ of non-negative integers.

\subsection{Hook-shaped SYT and skew-merged involutions}

\begin{definition}
We call a SYT a \emph{hook} if it consists of exactly one row and one column.
\label{def:hook}
\end{definition}

In the following, let $H_{n,k}$ denote the set of all hooks of size $n$ with  first row of length $k$.
The cardinality of $H_{n,k}$ will be denoted by $h_{n,k}$.

\begin{theorem}
For every positive integer $n$ the sequence $(h_{n,k})_{1\leq k \leq n}$ is log-concave.
\label{thm:log-conc_hooks}
\end{theorem} 

\begin{proof}
First, let us remark than it is straightforward to determine the numbers $h_{n,k}$.
Indeed, in order to create a hook with $n$ boxes with $k$ of them in the first row, we simply need to  choose the $(n-1)$ elements larger than 1 that will be in the first row.
The remaining elements are then placed in increasing order in the first column.
Thus
\[
h_{n,k}= \binom{n-1}{k-1}
\]
and it is immediately clear that the sequence $(h_{n,k})_{1\leq k \leq n}$ is log-concave.
In the following we will provide a combinatorial explanation for this fact.

In order to give a combinatorial proof of the log-concavity of $(h_{n,k})_{1\leq k \leq n}$ we follow the same procedure as in~\cite{bona2005combinatorial}: The sequence $(h_{n,k})_{1\leq k \leq n}$ is log-concave if and only if $h_{n,k} \cdot h_{n,l} \leq h_{n,k+1} \cdot h_{n,l-1}$ for all $n \geq 1$ and $1 \leq k \leq l-2 \leq n-2$.
We shall therefore inductively construct injections  $\varphi_{n,k,l}$ from $H_{n,k} \times H_{n,l}$ to $H_{n,k+1} \times H_{n,l-1}$ for all $n \geq 3$ and $1 \leq k \leq l-2 \leq n-2$. 
First we construct the injections $\varphi_{n,k,l}$ for the smallest
meaningful value of $n$, which is $n = 3$.
Since there is only a single element in $H_{3,1} \times H_{3,3}$, we shall also describe the functions $\varphi_{4,k,l}$ for all admissible values of $k$ and $l$.
Next, for the induction step, we use the assumption that the maps $\varphi_{n-1,k,l}$ exist for all admissible values of $k$ and $l$ to construct the maps $\varphi_{n,k,k+2}$.
It is not necessary to construct the maps $\varphi_{n,k,l}$ for $k <l-2$ since the existence of the injective maps $\varphi_{n,k,k+2}$ implies the log-concavity of the sequence $(h_{n,k})_{1\leq k \leq n}$ which implies the existence of the maps $\varphi_{n,k,l}$ for $1 \leq k < l-2 \leq n-2$.

First note that the element $n$ in a hook of size $n$ will always lie in the last box of the first column or in the last box of the first row.
This allows us to define the \emph{type} of a hook-shaped Standard Young Tableau: it is of type $\downarrow$ if the element lies in the first column and of type $\rightarrow$ otherwise.
Since we are dealing with pairs of SYT there are four possible types of pairs that can occur:
$\downarrow \downarrow $, $\downarrow \rightarrow$, $\rightarrow \downarrow$ and $\rightarrow \rightarrow$.
The maps $\varphi_{n,k,l}$ that we are going to construct are such that the type is preserved: the type of the image $\varphi_{n,k,l}(T_1, T_2)$ is the same as the type of $(T_1, T_2)$. 
This will allow us to prove the injectivity of these maps.

Now let us start with the base step at $n=3$ of our induction proof.
The map $\varphi_{3,1,3}$ is described in the top part of Figure~\ref{fig:hooks_induction_start}.
Since $\varphi_{3,1,3}$ is defined on a single element and this does not allow us to see what the functions $\varphi_{n,k,l}$ actually do, we have also included the description of the functions
$\varphi_{4,k,l}$ for $(k,l)=(1,3)$, $(1,4)$ and $(2,4)$ in the bottom part of Figure~\ref{fig:hooks_induction_start}.

\begin{figure}
\centering
\ytableausetup{aligntableaux=top,boxsize=1.1em
}
\begin{tabular}{P{2.3cm}|ccc}
$\mathbf{n=3}$ $(k,l)=(1,3)$ & 
$\begin{pmatrix} \, \begin{ytableau}
1   \\
2 \\
*(black!20) 3
\end{ytableau} \, , \,
\begin{ytableau}
1 & 2 &  *(black!20) 3
\end{ytableau} \, \,\end{pmatrix}$  & $\longmapsto$ 
& 
$\begin{pmatrix} \, \begin{ytableau}
1  & 2 \\
*(black!20) 3
\end{ytableau} \, , \,
\begin{ytableau}
1 & *(black!20) 3 \\
2
\end{ytableau} \, \,\end{pmatrix}$ \\
[5ex] \hline \hline \\ [-1.5ex]
$\mathbf{n=4}$ $(k,l)=(1,3)$ & 
$\begin{pmatrix} \, \begin{ytableau}
1   \\
2 \\
3 \\
*(black!20) 4
\end{ytableau} \, , \,
\begin{ytableau}
1 & 2 &   3 \\
*(black!20) 4
\end{ytableau} \,\,\end{pmatrix}$  & $\longmapsto$ 
& 
$\begin{pmatrix} \, \begin{ytableau}
1 & 2 \\
3 \\
*(black!20) 4
\end{ytableau} \, , \,
\begin{ytableau}
1 & 3 \\
2 \\
*(black!20) 4
\end{ytableau} \, \,\end{pmatrix}$ \\ \\
& $\begin{pmatrix} \, \begin{ytableau}
1   \\
2 \\
3 \\
*(black!20) 4
\end{ytableau} \, , \,
\begin{ytableau}
1 & 2 &   *(black!20) 4 \\
3
\end{ytableau} \,\,\end{pmatrix}$  & $\longmapsto$ 
& 
$\begin{pmatrix} \, \begin{ytableau}
1 & 2 \\
3 \\
*(black!20) 4
\end{ytableau} \, , \,
\begin{ytableau}
1 & *(black!20) 4 \\
2 \\
3
\end{ytableau} \, \,\end{pmatrix}$ \\ \\
& $\begin{pmatrix} \, \begin{ytableau}
1   \\
2 \\
3 \\
*(black!20) 4
\end{ytableau} \, , \,
\begin{ytableau}
1 & 3 &   *(black!20) 4 \\
2
\end{ytableau} \,\,\end{pmatrix}$  & $\longmapsto$ 
& 
$\begin{pmatrix} \, \begin{ytableau}
1 & 3 \\
2 \\
*(black!20) 4
\end{ytableau} \, , \,
\begin{ytableau}
1 & *(black!20) 4 \\
2 \\
3
\end{ytableau} \, \,\end{pmatrix}$ \\ \\
\hline \\ [-1.5ex]
$(k,l)=(1,4)$ & $\begin{pmatrix} \, \begin{ytableau}
1   \\
2 \\
3 \\
*(black!20) 4
\end{ytableau} \, , \,
\begin{ytableau}
1 & 2 &   3 & *(black!20) 4
\end{ytableau} \,\,\end{pmatrix}$  & $\longmapsto$ 
& 
$\begin{pmatrix} \, \begin{ytableau}
1 & 2 \\
3 \\
*(black!20) 4
\end{ytableau} \, , \,
\begin{ytableau}
1 & 2 & *(black!20) 4 \\
3 
\end{ytableau} \, \,\end{pmatrix}$ \\ \\
\hline \\ [-1.5ex]
$(k,l)=(2,4)$ & $\begin{pmatrix} \, \begin{ytableau}
1 & 2  \\
3 \\
*(black!20) 4
\end{ytableau} \, , \,
\begin{ytableau}
1 & 2 &   3 & *(black!20) 4
\end{ytableau} \,\,\end{pmatrix}$  & $\longmapsto$ 
& 
$\begin{pmatrix} \, \begin{ytableau}
1 & 2 & 3\\
*(black!20) 4
\end{ytableau} \, , \,
\begin{ytableau}
1 & 2 & *(black!20) 4 \\
3 \\
\end{ytableau} \, \,\end{pmatrix}$ \\ \\
& $\begin{pmatrix} \, \begin{ytableau}
1  & *(black!20) 4  \\
2 \\
3
\end{ytableau} \, , \,
\begin{ytableau}
1 & 2 &   3 & *(black!20) 4\\
\end{ytableau} \,\,\end{pmatrix}$  & $\longmapsto$ 
& 
$\begin{pmatrix} \, \begin{ytableau}
1 & 2 & *(black!20) 4\\
3
\end{ytableau} \, , \,
\begin{ytableau}
1 & 3 & *(black!20) 4\\
2 \\
\end{ytableau} \, \,\end{pmatrix}$ \\ \\
& $\begin{pmatrix} \, \begin{ytableau}
1  & 3 \\
2 \\
*(black!20) 4
\end{ytableau} \, , \,
\begin{ytableau}
1 & 2 &   3 &  *(black!20) 4\\
\end{ytableau} \,\,\end{pmatrix}$  & $\longmapsto$ 
& 
$\begin{pmatrix} \, \begin{ytableau}
1 & 2 & 3\\
*(black!20) 4
\end{ytableau} \, , \,
\begin{ytableau}
1 & 3 & *(black!20) 4\\
2 \\
\end{ytableau} \, \,\end{pmatrix}$ \\
\end{tabular}

\caption{The base step of the induction in the proof of Theorem~\ref{thm:log-conc_hooks}: the maps $\varphi_{n,k,l}$ for $n=3$ and $n=4$ and all allowed values for $k$ and $l$.
The box containing the largest element is marked in gray in every SYT and we can see that a pair of SYT of a given type is always mapped to a pair of SYT of the same type. }
\label{fig:hooks_induction_start}
\end{figure}

Let us turn to the induction step and assume that the injective functions $\varphi_{n-1,k,l}$ for $1 \leq k  \leq l-2 \leq n-3$ have already been constructed.
The definition of the function $\varphi_{n,k,k+2}$ depends on the type of the pair of Tableaux it is applied to and we have two different rules:
\begin{enumerate}
\item Type $\downarrow \rightarrow$:

This is the easy case: For a pair $(T_1,T_2) \in H_{n,k} \times H_{n,k+2}$ of type $\downarrow \rightarrow$ we can take the element $n$ in $T_1$ and move it to the end of the first row in $T_1$ in order to obtain a Tableau $U_1$ where the first row has length $k+1$. 
Similarly, in $T_2$ we can take the element $n$ and move it to the end of the first column in order to obtain a Tableau $U_2$ where the first row has length $k+1$. 
Now the type of $(U_1, U_2)$ is $\rightarrow \downarrow$, so we define $\varphi_{n,k,l}(T_1, T_2)$ to be $(U_2,U_1)$.
For an example of the map $\varphi_{n,k,k+2}$ in this case for $n=5$, see Figure~\ref{fig:example_hooks_n=5}.

\item Other type:

When the type is not $\downarrow \rightarrow$ it is less obvious how to define the map $\varphi_{n,k,k+2}$.
Here we make use of the maps $\varphi_{n-1,k,l}$ for $1 \leq k  \leq l-2 \leq n-3$ that exist by induction hypothesis.
The function $\varphi_{n,k,k+2}$ is then defined as follows for a pair $(T_1,T_2) \in H_{n,k} \times H_{n,k+2}$: First we remove the element $n$ both in $T_1$ and in $T_2$. 
If the type is $\downarrow \downarrow $ we obtain a pair $(t_1,t_2) \in H_{n-1,k} \times H_{n-1,k+2}$, if it is $\rightarrow \downarrow$ we obtain a pair $(t_1,t_2) \in H_{n-1,k-1} \times H_{n-1,k+2}$ and if it is $\rightarrow \rightarrow$ we obtain a pair $(t_1,t_2) \in H_{n-1,k-1} \times H_{n-1,k+1}$.
In all three cases we can apply one of the maps $\varphi_{n-1,k,l}$ with  $1 \leq k \leq l-2 \leq n-3$ for suitable values of $k$ and $l$.
We do so and obtain a pair $(u_1,u_2)$ which is in $H_{n-1,k+1} \times H_{n-1,k+1}$ for type $\downarrow \downarrow$, in $H_{n-1,k} \times H_{n-1,k+1}$ for type $\rightarrow \downarrow$ and in  $H_{n-1,k} \times H_{n-1,k}$ for type $\rightarrow \rightarrow$.
Finally, we replace the element $n$ in both Tableaux of the pair $(u_1,u_2)$ according to its original positions in $(T_1,T_2)$, thus creating a pair of Tableaux with $n$ boxes of the same type as $(T_1,T_2)$.
For all three types, replacing the element $n$ in its original position will lead to a pair $(U_1, U_2) \in H_{n,k+1} \times H_{n,k+1}$.
For an example of the map $\varphi_{n,k,k+2}$ for a pair of type $\rightarrow \rightarrow$ and  for $n=5$, see Figure~\ref{fig:example_hooks_n=5}.
\end{enumerate}
The construction described above ensures that the type of the image $(U_1,U_2)$ under the map $\varphi_{n,k,k+2}$ will always be the same as the one of $(T_1,T_2)$.
Thus, in order to prove that the map $\varphi_{n,k,k+2}$ is injective, it suffices to prove that two distinct pairs $(T_1, T_2)$ and $(S_1,S_2)$ of SYT with $n$ boxes that are of the same type cannot have the same image.
First, let us take a look at the case of pairs of type $\downarrow \rightarrow$: If the image of two pairs $(T_1, T_2)$ and $(S_1,S_2)$ of type $\downarrow \rightarrow$ is the same, this means that the two pairs have to be the same if we remove the element $n$ in all involved SYT. 
Since $(T_1, T_2)$  and $(S_1,S_2)$ are both of type $\downarrow \rightarrow$ it follows that $(T_1, T_2)=(S_1,S_2)$.
Second, the case of pairs of other type: the argument is similar here. 
For two pairs $(T_1, T_2)$ and $(S_1,S_2)$ let us denote by $(t_1,t_2)$ and $(s_1,s_2)$ the corresponding pairs of SYT where the element $n$ has been removed.
If the image of $(T_1, T_2)$ and $(S_1,S_2)$ is the same, this means that the image of $(t_1,t_2)$ and $(s_1,s_2)$ has to be the same.
The image of $(t_1,t_2)$ and $(s_1,s_2)$ is given by one of the maps $\varphi_{n-1,k,l}$ for suitable $k$ and $l$.
By the induction hypothesis these maps are injective and thus the image of $(t_1,t_2)$ and $(s_1,s_2)$ can only be the same if $(t_1,t_2)=(s_1,s_2)$
This in turn implies that $(T_1, T_2)=(S_1,S_2)$.
This finishes the proof.

\begin{figure}
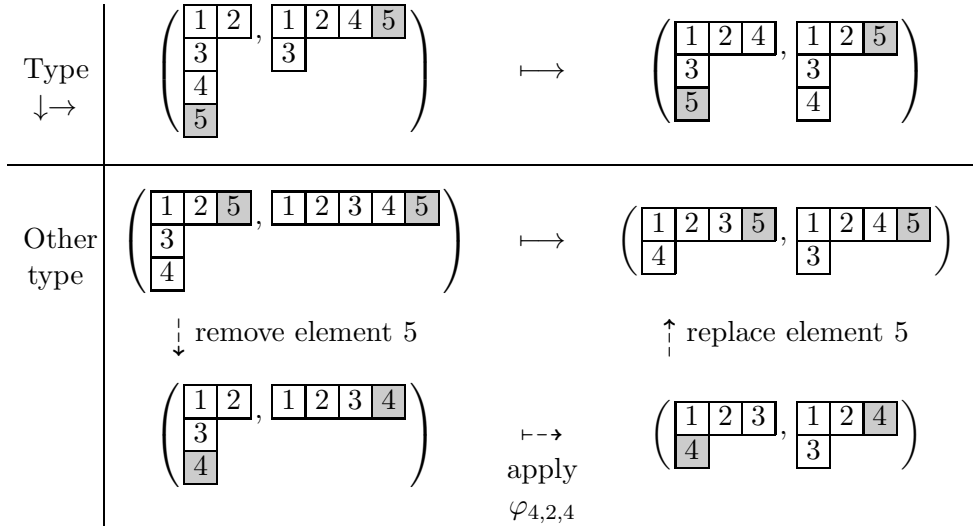

\centering
\ytableausetup{aligntableaux=top,boxsize=1.1em
}
\begin{tabular}{P{0.85cm}|cP{1.1cm}c}
Type $\da \ra$ & $\begin{pmatrix} \, \begin{ytableau}
1 & 2  \\
3 \\
4 \\
*(black!20) 5
\end{ytableau} \, , \,
\begin{ytableau}
1 & 2 & 4 & *(black!20) 5 \\
3
\end{ytableau} \, \,\end{pmatrix}$  & $\longmapsto$ 
& 
$\begin{pmatrix} \, \begin{ytableau}
1  & 2  & 4\\
3 \\
*(black!20) 5
\end{ytableau} \, , \,
\begin{ytableau}
1 & 2 & *(black!20) 5 \\
3 \\
4
\end{ytableau} \, \, \end{pmatrix}$ \\
[6ex] \hline \\[-1ex]

Other type & $\begin{pmatrix} \, \begin{ytableau}
1 & 2 & *(black!20) 5\\
3 \\
4 \\
\end{ytableau} \, , \,
\begin{ytableau}
1 & 2 & 3 & 4 & *(black!20) 5 \\
\end{ytableau} \,\,\end{pmatrix}$ & $\longmapsto$ 
& 
$\begin{pmatrix} \, \begin{ytableau}
1  & 2  & 3 & *(black!20) 5 \\
4 \\
\end{ytableau} \, , \,
\begin{ytableau}
1 & 2 & 4 & *(black!20) 5 \\
3 \\
\end{ytableau} \, \,\end{pmatrix}$ \\
[4.5ex]
&\textcolor{black}{$\dashdownarrow$ remove element 5} &&   \textcolor{black}{$\dashuparrow$ replace element 5}\\ [2.5ex]
& \textcolor{black}{
$\begin{pmatrix} \, \begin{ytableau}
1 & 2 \\
3 \\
*(black!20) 4 \\
\end{ytableau} \, , \,
\begin{ytableau}
1 & 2 & 3 &  *(black!20) 4 \\
\end{ytableau} \, \,\end{pmatrix}$}  & \textcolor{black}{$\dashmapsto$ apply  $\varphi_{4,2,4}$}
& \textcolor{black}{
$\begin{pmatrix} \, \begin{ytableau}
1  & 2  & 3  \\
*(black!20) 4 \\
\end{ytableau} \, , \,
\begin{ytableau}
1 & 2 & *(black!20) 4 \\
3 \\
\end{ytableau} \, \,\end{pmatrix}$} \\
\end{tabular}
\caption{Two examples of the injective map $\varphi_{5,k,l}$ described in the proof of Theorem~\ref{thm:log-conc_hooks}.}
\label{fig:example_hooks_n=5}
\end{figure}
\end{proof}

Now let us turn from SYT to permutations: 
Pairs of hooks of size $n$ and with a first row of length $k$ bijectively correspond to permutations of length $n$ and with a longest increasing sequence of length $k$ and a longest decreasing sequence of length $n-k+1$.
That is, these permutations are the merge of an increasing and a decreasing subsequence having one point in common.
Let us denote these numbers by $m_{n,k}$.
Then Theorems~\ref{thm:involutions_to_arbirary_permutations} and~\ref{thm:log-conc_hooks} lead to the follwoing:

\begin{corollary}
For every positive integer $n$ the sequence $(m_{n,k})_{1 \leq k \leq n}$ is log-concave.
\end{corollary}

Permutations which are the merge of an increasing and a decreasing sequence are called \textit{skew-merged} permutations.
They have been shown to be exactly those permutations avoiding the two patterns $2143$ and $3412$ \cite{stankova1994forbidden}. 
However, not all skew-merged permutations correspond to hook-shaped SYT.
Indeed, a skew-merged permutation of length $n$ can consist of a longest increasing subsequence of length $k$ and of a longest decreasing subsequence of length $n-k$, i.e., the two sequences do not intersect.
The shape of the corresponding tableau is then not a hook but a hook with an additional box at position $(2,2)$.
Note however that not all pairs of SYT of this shape actually correspond to skew-merged permutations.

For skew-merged involutions the situation appears to be simpler and we can prove the following result:

\begin{proposition}
The SYT associated to a skew-merged involution is always hook-shaped.
\label{thm:SMinvol_hook-shaped}
\end{proposition}

From this we finally obtain the following result about skew-merged involutions:

\begin{corollary}
The number of skew-merged involutions of length $n$ and with longest increasing sequence of length $k$ is:
\[
i_{n,k}= \binom{n-1}{k-1}.
\]
Thus the total number $i_n$ of skew-merged involutions is simply $2^{n-1}$ and the sequence $(i_{n,k})_{1 \leq k \leq n}$ is log-concave.
\end{corollary}

\begin{proof}[Proof of Proposition \ref{thm:SMinvol_hook-shaped}]
In the following, we will use the notation introduced by Atkinson in~\cite{atkinson1998permutations} and will apply one of the intermediary Lemmas proven there.

Atkinson views a permutation $\sigma$ as the set of points  $(i,\sigma(i))$ in the plane.
If $\sigma$ is skew-merged the points can be partitioned into  five (possibly empty) sets: there are red, blue, green, yellow and white points as represented in Figure~\ref{fig:colors_skew_merged}.
The red and yellow points are decreasing, the green and blue ones are increasing and the white points are either increasing or decreasing.

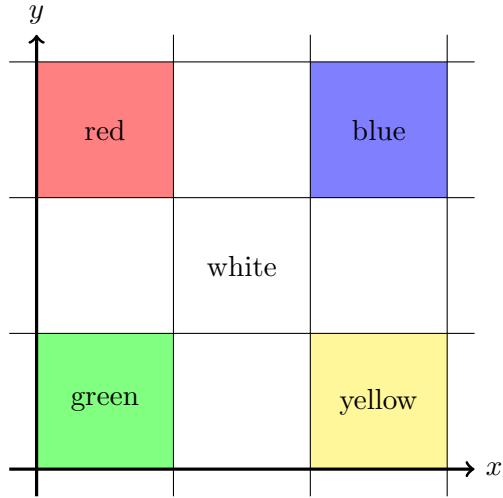
\begin{figure}
\begin{center}
\begin{tikzpicture}[scale=0.6,x=3cm,y=3cm]

  \def\xmin{-.2}
  \def\xmax{3.2}
  \def\ymin{-0.2}
  \def\ymax{3.2}
  
  \fill [green!50] (0,0) rectangle (1,1);
\fill [red!50] (0,2) rectangle (1,3);
\fill [yellow!50] (2,0) rectangle (3,1);
\fill [blue!50] (2,2) rectangle (3,3);
  
  \draw[->, very thick] (\xmin,0) -- (\xmax,0) node[right] {$x$};
  \draw[->, very thick] (0,\ymin) -- (0,\ymax) node[above] {$y$};
  
  \draw[-] (\xmin,1) -- (\xmax,1);
  \draw[-] (\xmin,2) -- (\xmax,2);
  \draw[-] (\xmin,3) -- (\xmax,3);
  \draw[-] (1,\ymin) -- (1,\ymax);
  \draw[-] (2,\ymin) -- (2,\ymax);
  \draw[-] (3,\ymin) -- (3,\ymax);

\node at (0.5,0.5) {green};
\node at (0.5,2.5) {red};
\node at (1.5,1.5) {white}; 
\node at (2.5,0.5) {yellow};
\node at (2.5,2.5) {blue};
\end{tikzpicture}
\end{center}
\caption{The disposition of colors in a skew-merged permutation due to Atkinson~\cite{atkinson1998permutations}.}
\label{fig:colors_skew_merged}
\end{figure}

The points that are of particular interest to us are the white ones.
White points can be defined as follows: Whenever one chooses two points $(i,r)$ and $(j,s)$ that are not of the same color and are to the left (or to the right) of a white point $(k,t)$, i.e. $i,j<k$ (or $k> i,j$), $t$ is neither the largest nor the smallest among the elements $r, s$ and $t$.
Skew-merged permutations with at least one white point are exactly those which are the union of an increasing subsequence $\alpha$ and a decreasing subsequence $\beta$ that have a common point.
Thus skew-merged permutations with at least one white point correspond to hook-shaped SYT.
Conversely skew-merged permutations with no white points are those where the associated SYT is not  a hook but a hook with an additional box at position $(2,2)$.

The goal of this proof is to show that a skew-merged permutation with no white elements cannot be an involution.

For this, we assume that the skew-merged involution $\sigma$ has no white elements and construct a contradiction to the fact that $\sigma$ is skew-merged.

In order to so, we will need a slightly weaker version of one of the two assertions in Lemma 11 in~\cite{atkinson1998permutations}. 
We shall state it here in the form we need it:
\begin{quote} \textbf{Lemma} \cite{atkinson1998permutations}:
\emph{Suppose that $\sigma$ is a skew-merged permutation of length $n$  with no white points. 
Then there exist indices $1 \leq i < j< j+1< k \leq n$ such that one of the following two statements holds:
\begin{itemize}
\item $\sigma(i) \sigma (j) \sigma (j+1) \sigma (k)$ forms a 3142-pattern in $\sigma$
\item $\sigma(i) \sigma (j) \sigma (j+1) \sigma (k)$ forms a 2413-pattern in $\sigma$
\end{itemize}  }
\end{quote}

We will now show the following: If $\sigma(i) \sigma (j) \sigma (j+1) \sigma (k)$ forms a 3142-pattern then $\sigma$ also contains the pattern 3412 or 2143 and is thus not skew-merged.
In order to do so we distinguish three different cases.
To alleviate notation, we will write $cadb$ instead of $\sigma(i) \sigma (j) \sigma (j+1) \sigma (k)$ and will keep in mind that $a<b<c<d$.
\begin{enumerate}
\item $d \leq j+1$:

This implies that $a \leq d-3 \leq j-2$ and that $b \leq d-2 \leq j-1$.
Especially this means that $a$ and $b$ cannot be fixed points in $\sigma$.
Since $\sigma$ is an involution we thus have $\sigma(a)=j$ and $\sigma(b)=k$.
Since $a < b< j<k$, we have $jkab$ as a subsequence of $\sigma$ that forms a $3412$-pattern.
\item $a \geq j$:

This implies that $c \geq a+2 \geq j+2$ and that $d \geq a+3 \geq j+3$.
Especially this means that $c$ and $d$ cannot be fixed points in $\sigma$.
Since $\sigma$ is an involution we thus have $\sigma(c)=i$ and $\sigma(d)=j+1$.
Since $i < j+1 < c< d$, we have $cda(j+1)$ as a subsequence of $\sigma$ that forms a $3412$-pattern.
\item$d > j+1$ and $a <j$:

In this case $a$ and $d$ aren't fixed points and since $\sigma$ is an involution we have $\sigma(a)=j$ and $\sigma(d)=j+1$.
Since $a<j<j+1<d$, we have $j a d (j+1)$ as a subsequence of $\sigma$ that forms a $2143$-pattern.
\end{enumerate}
Note that it is crucial in the arguments above that the elements $a$ and $d$ are adjacent in $\sigma$ which is due to the fact that there are no white points in $\sigma$.

If the second statement of the Lemma above is fulfilled, that is if $\sigma(i) \sigma (j)$ $\sigma (j+1) \sigma (k)$ forms a 2413-pattern, we consider the reversed permutation $\sigma ^r$, i.e., the permutation $\sigma$ read from right to left.
The permutation $\sigma ^r$ then contains the pattern 3142 at the positions $(n+1-k), (n-j), (n+1-j)$ and $(n+1-i)$. 
This implies that $\sigma ^r$ also contains 3412 or 2143 as a pattern and is not skew-merged.
Remark that a permutation is skew-merged exactly then when its reverse is skew-merged.
Thus we conclude that $\sigma$ is not skew-merged in this case either.

We have thus proven that a skew-merged involution must contain at least one white point and the shape of its associated SYT is a hook.
\end{proof}

\subsection{The set of $(l,m)$-protected SYT}

\begin{definition}
Given an arbitrary SYT $T$, it can be decomposed into two parts in a unique way as follows:
The \emph{protected area} of $T$ is obtained from $T$ by removing as many boxes as possible in the first row and in the first column of $T$ without creating a shape that is no longer a Ferrers diagram.
The removed elements form the \emph{surplus} of $T$.
The elements that have been removed in the first row are referred to as the \emph{eastern surplus} and the ones in the first column as the \emph{southern surplus}.
\label{def:protected_hook}
\end{definition}

For an illustration of this decomposition of a SYT into \textit{protected area} and \textit{surplus}, see Figure~\ref{fig:protected_hook}.
Note that if $T$ is a hook in the sense of Definition~\ref{def:hook}, then the protected area of $T$ consists of the box containing the element $1$ only.

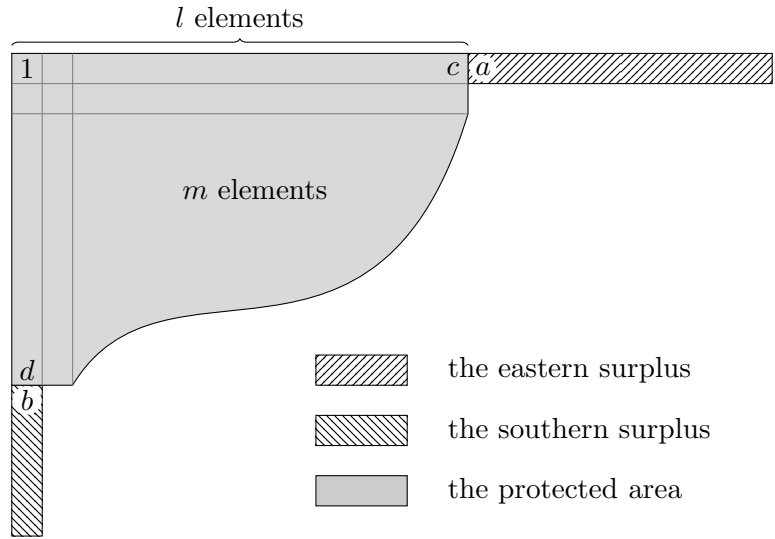
\begin{figure}
\begin{center}
\begin{center}
\begin{tikzpicture}[scale=0.4]
\fill[pattern=north east lines, pattern color=black, draw=black] (15,0) rectangle (25,1);
\fill[pattern=north west lines, pattern color=black, draw=black] (0,-15) rectangle (1,-10);
\fill[black!15, draw=black] (0,1) -- (15,1) -- (15,-1) .. controls (12,-11) and (5,-5) .. (2,-10) -- (0,-10)  -- cycle;
\node at (0.5,0.5) {1};
\draw[draw=black!50] (0,0) -- (15,0);
\draw[draw=black!50] (0,-1) -- (15,-1);
\draw[draw=black!50] (1,1) -- (1,-10);
\draw[draw=black!50] (2,1) -- (2,-10);
\node at (8,-3.5) {$m$ elements};
\draw [decorate,decoration={brace,mirror,raise=3pt}] (15,1) -- (0,1) node [midway,above=6pt] {$l$ elements};

\fill[pattern=north east lines, pattern color=black, draw=black] (10,-10) rectangle (13,-9);
\node[anchor=west] at (14,-9.5) {the eastern surplus};
\fill[pattern=north west lines, pattern color=black, draw=black] (10,-12) rectangle (13,-11);
\node[anchor=west] at (14,-11.5) {the southern surplus};
\fill[black!20, draw=black] (10,-14) rectangle (13,-13);
\node[anchor=west] at (14,-13.5) {the protected area};

\node[circle, fill=white, inner sep=0.3mm] at (15.5,0.5) {$a$};
\node at (14.5,0.5) {$c$};

\node at (0.5,-9.5) {$d$};
\node[circle, fill=white, inner sep=0.25mm] at (0.5,-10.5) {$b$};

\end{tikzpicture}
\end{center}
\end{center}
\caption{The decomposition of a SYT into its \emph{surplus} and \emph{protected area}. The protected area consists of $m$ elements of which $l$ are in the first row.}
\label{fig:protected_hook}
\end{figure}

We can now define the following class of SYT:

\begin{definition}
Let $T$ be a SYT in which all elements contained in its surplus are larger than all elements contained in the first row and first column of its protected area. 
If $T$ has a protected area of size $m$ of which $l$ elements are contained in the first row it is called \emph{$(l,m)$-protected}.
\end{definition}

For an example of an $(l,m)$-protected SYT with $l=4$ and $m=12$, see the top part of Figure~\ref{fig:example_protected} where the protected areas have been marked in gray.

Let $a$, $b$, $c$ and $d$ be the elements of $T$ as displayed in Figure~\ref{fig:protected_hook}, i.e., $a$ is the smallest element in the eastern surplus and $c$ is its left neighbour in $T$, $b$ is the smallest element in the southern surplus and $d$ is its top neighbour in $T$.
Then the condition that all elements in the surplus are larger than those in the first row and first column of the protected area is equivalent to:
\begin{align}
\min(a,b) > \max(c,d).
\label{eqn:condition_protected}
\end{align}

Before we tackle our conjecture for the set of $(l,m)$-protected SYT, let us consider the special case of $(2,4)$-protected SYT and take a closer look at the set of these SYT.
For $(2,4)$-protected SYT, i.e., SYT whose shape is a hook with an additional box at the position $(2,2)$, condition~\eqref{eqn:condition_protected} is fulfilled if and only if the elements $1$, $2$, and $3$ are contained in the protected area.
Equivalently, the elements $1$, $2$, and $3$ may not be contained in the same row or column of the SYT.
Turning to involutions, this translates as follows:
Involutions that correspond to $(2,4)$-protected SYT have the property that the lengths of the longest increasing and of the longest decreasing sequence add up to the total length of the permutation and the elements 1, 2, and 3 do not form a 123- or a 321- pattern.

A question that is of interest here is the following: how many SYT whose shape is a hook with an additional box at the position $(2,2)$ actually fulfil condition~\eqref{eqn:condition_protected}?
We will see that this is the case for roughly half of the SYT of this shape.
\begin{proposition}
Let $p_n$ denote the number of $(2,4)$-protected SYT of size $n$ and let $b_n$ denote the number of SYT whose shape is a hook with an additional box at the position $(2,2)$.
Then the following holds for all $n \geq 4$:
\[
p_n=(n-3)2^{n-3}, \, b_n= (n-4)2^{n-2}+2 \text{ and thus} \lim_{n \to \infty}\frac{p_n}{b_n}= \frac{1}{2}.
\]
\end{proposition} 

\begin{proof}
Let us start by counting $(2,4)$-protected SYT.
We know that the elements $1$, $2$ and $3$ have to be present in the protected area and there are two different possibilities of arranging them.
Moreover we can choose the element $i$ that will lie in the box at position $(2,2)$ among the integers $4, \ldots, n$.
The remaining entries are then inserted one after the other in increasing order.
For every element we can choose to place it either in the eastern  or in the southern surplus of the tableau which gives us a total of $2^{n-4}$ possibilities. 
In total we have:
\[
p_n=2 \cdot (n-3)2^{n-4}=(n-3)2^{n-3}.
\]%

Now let us count SYT whose shape is a hook with an additional box at the position $(2,2)$ but that are not $(2,4)$-protected.
That is, we determine $b_n-p_n$.
This means that the elements $1$, $2$ and $3$ are all contained in the first row or in the first column of the SYT.
Let us concentrate on the case where $1$, $2$ and $3$ are all contained in the first row; the second case can then be obtained by symmetry.
Let us denote by $i$ the element at position $(2,1)$, i.e., to the south of $1$, and by $j$ the element at position $(2,2)$, i.e., to the east of $i$ and to the south of $2$.
The element $i$ can be any integer in $\left\lbrace 4, \ldots, n-1\right\rbrace$ and $j$ can be any integer in $\left\lbrace i+1, \ldots, n\right\rbrace$.
All elements that are smaller than $i$ have to be placed in the eastern surplus.
For the elements that are larger than $i$ and not equal to $j$, we can choose whether to place them in the eastern or southern surplus. There are thus $2^{n-i-1}$ possibilities of placing these elements.
In total we have:
\[
b_n-p_n= 2 \cdot  \sum_{i=4}^{n-1} (n-i) \cdot 2^{n-i-1},
\]
which leads to:
\begin{align*}
b_n & = \sum_{i=3}^{n-1} (n-i) \cdot 2^{n-i}  = 2^n \cdot \left( n \cdot \sum_{i=3}^{n-1} 2^{-i} - \sum_{i=3}^{n-1} i2^{-i} \right) \\
& = 2^n \cdot \left( n \cdot \left(\frac{1}{4}-\frac{1}{2^n}\right) - \frac{2^n -2-n}{2^n}  \right) = 2^{n-2}(n-4) +2. 
\end{align*}
We indeed obtain that $p_n$ is roughly one half of $b_n$ and that the fraction $p_n/b_n$ tends to $1/2$ when $n$ tends to infinity.
\end{proof}
Note that this result for the numbers $b_n$ can of course also be obtained by applying the hooklength formula (see Chapter 14 in \cite{handbook}).
However this leads to rather tedious manipulations of sums involving fractions of binomial coefficients and the approach above is much faster.

\medskip
Now let us turn to $(l,m)$-protected SYT for arbitrary integers $l$ and $m$. 
In the following, for $1 \leq m \leq n$ and $1 \leq l \leq k$, we will denote by $P_{n,k}^{(l,m)}$ the set of all $(l,m)$-protected SYT of size $n$ where the first row has length $k$.
Whenever it is clear from the context, we will write $P_{n,k}$ instead of $P_{n,k}^{(l,m)}$.
The cardinality of $P_{n,k}^{(l,m)}$ shall be denoted by $p_{n,k}^{(l,m)}$ or $p_{n,k}$. 

\begin{theorem}
For every positive integer $n$ and every fixed pair $(l,m)$, the sequence $(p_{n,k})_{1 \leq k \leq n}$ is log-concave.
\label{thm:log-concave_protected}
\end{theorem}

\begin{figure}
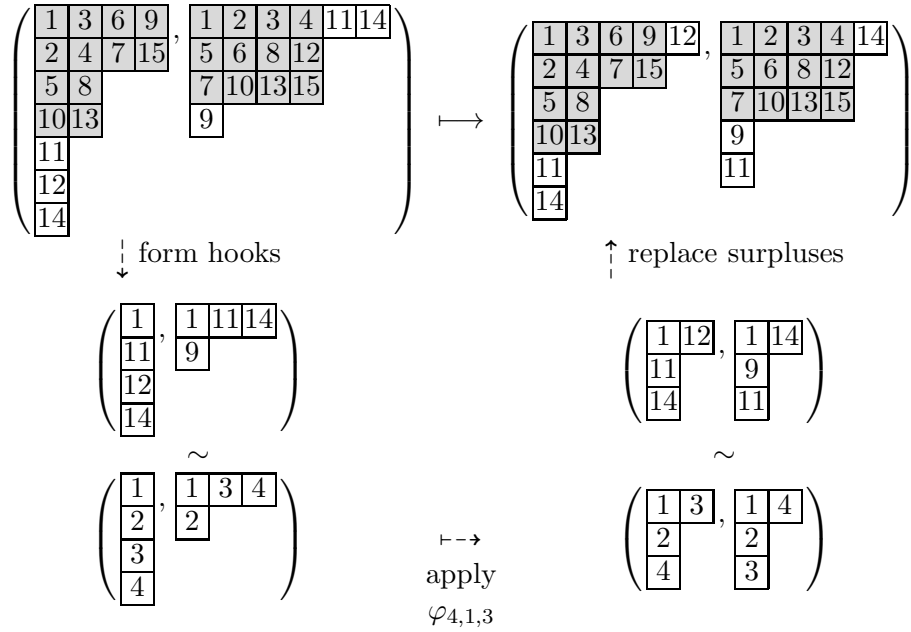

\centering
\ytableausetup{aligntableaux=top,boxsize=1.1em
}
\begin{tabular}{cP{1cm}c}
 $\begin{pmatrix} \, \begin{ytableau}
*(black!15)1 & *(black!15) 3 & *(black!15) 6 & *(black!15) 9\\
*(black!15) 2 & *(black!15)4 & *(black!15)7 & *(black!15)15 \\
*(black!15)5 & *(black!15)8 \\
*(black!15)10 & *(black!15)13 \\
11 \\
12 \\
14
\end{ytableau} \, , \,
\begin{ytableau}
*(black!15)1 & *(black!15) 2 & *(black!15) 3 & *(black!15) 4 & 11 & 14\\
*(black!15) 5 & *(black!15)6 & *(black!15)8 & *(black!15)12 \\
*(black!15)7 & *(black!15)10 &*(black!15)13 & *(black!15)15 \\
9 \\
\end{ytableau} \,\,\end{pmatrix}$ \hspace{-0.5cm} & $\longmapsto$ 

&\hspace{-0.5cm} $\begin{pmatrix} \, \begin{ytableau}
*(black!15)1 & *(black!15) 3 & *(black!15) 6 & *(black!15) 9 & 12\\
*(black!15) 2 & *(black!15)4 & *(black!15)7 & *(black!15)15 \\
*(black!15)5 & *(black!15)8 \\
*(black!15)10 & *(black!15)13 \\
11 \\
14
\end{ytableau} \, , \,
\begin{ytableau}
*(black!15)1 & *(black!15) 2 & *(black!15) 3 & *(black!15) 4  & 14\\
*(black!15) 5 & *(black!15)6 & *(black!15)8 & *(black!15)12 \\
*(black!15)7 & *(black!15)10 &*(black!15)13 & *(black!15)15 \\
9 \\
11
\end{ytableau} \,\,\end{pmatrix}$ \\
[4.5ex]
\textcolor{black}{$\dashdownarrow$ form hooks} &&   \textcolor{black}{$\dashuparrow$ replace surpluses}\\ [2.5ex]
$\begin{pmatrix} \, \begin{ytableau}
1 \\
11 \\
12 \\
14
\end{ytableau} \, , \,
\begin{ytableau}
1  & 11 & 14\\
9 \\
\end{ytableau} \, \,\end{pmatrix}
$ & & 
$
\begin{pmatrix} \, \begin{ytableau}
1 & 12 \\
11 \\
14
\end{ytableau} \, , \,
\begin{ytableau}
1 &  14 \\
9 \\
11
\end{ytableau} \, \,\end{pmatrix}
$ \\
$\sim$ && $\sim$ \\

$\begin{pmatrix} \, \begin{ytableau}
1 \\
2 \\
3 \\
4
\end{ytableau} \, , \,
\begin{ytableau}
1  & 3 &  4\\
2 \\
\end{ytableau} \, \,\end{pmatrix}$
 & \textcolor{black}{$\dashmapsto$ apply  $\varphi_{4,1,3}$}
& $\begin{pmatrix} \, \begin{ytableau}
1 & 3 \\
2 \\
 4
\end{ytableau} \, , \,
\begin{ytableau}
1 &  4 \\
2 \\
3
\end{ytableau} \, \,\end{pmatrix}
$ \\
\end{tabular}
\caption{An example of the injective map $\psi_{15,5}$ described in the proof of Theorem~\ref{thm:log-conc_hooks}. Here, the map is applied to a pair of $(4,12)$-protected SYT.}
\label{fig:example_protected}
\end{figure}

\begin{proof}
Let us fix the integers $l$ and $m$ and omit them in the notation.
In order to prove the log-concavity of the sequence $(p_{n,k})$ we need to construct injections $\psi_{n,k}$ from $P_{n,k-1} \times P_{n,k+1}$ to $P_{n,k} \times P_{n,k}$ for all $n \in \mathbb{N}$ and $1 \leq k \leq n$.

Let $(T_1, T_2)$ be a pair of SYT in $P_{n,k-1} \times P_{n,k+1}$. 
The maps $\psi_{n,k}$ shall never affect the protected areas of $T_1$ and $T_2$ -- thus this name.
Let $H_1$ and $H_2$ be the hooks that one obtains from the surpluses of $T_1$ and $T_2$ as follows: Place the element 1 at position $(1,1)$; then attach the respective eastern surplus to the east of $1$ and the southern surplus to the south of $1$.
The hooks $H_1$ and $H_2$ then both consist of $(n-m+1)$ boxes.
Moreover, the first row of $H_1$ has $k-l$ elements and the first row of $H_2$ has $k-l+2$ elements.
Let us denote by $\tilde{H_i}$  the SYT that we obtain from $H_i$ for $i=1,2$ by replacing the entries order-isomorphically by the integers from 1 up to $(n-m+1)$.
We can then apply the injection $\varphi_{n-m+1,k-l,k-l+2}$ defined in the proof of Theorem~\ref{thm:log-conc_hooks} to $(\tilde{H}_1,\tilde{H}_2)$ and obtain a pair $(J_1,J_2)$ in $H_{n-m+1, k-l+1} \times H_{n-m+1,k-l+1}$. 

In order to obtain $(U_1,U_2)$, the image of $(T_1, T_2)$ under $\psi_{n,k}$, we now do the following for $i=1,2$:
First we create the hooks $\tilde{J_i}$ by order-isomorphically replacing the elements $2, \ldots, n-m+1$ in $J_i$ by the elements that occurred in the respective surpluses of $T_i$.
Then we attach the eastern surplus of $\tilde{J_i}$ to the east of the first row of the protected area of $T_i$ and the southern surplus of $\tilde{J_i}$ to the south of its first column.

The only change that has been made to the shapes of $T_1$ and $T_2$ was to remove one box at the end of the first column or row and to place it at the end of the first row or column.
The shapes of $U_1$ and $U_2$ are thus as desired and the sizes of the protected areas are unchanged.
We still need to check whether $U_1$ and $U_2$ are actually SYTs, i.e., whether the numbers in all rows and columns are increasing.
Clearly, we only need to do so for the first row and the first column since the other parts have not been affected. 
By definition of the maps $\varphi_{n-m+1,k-l,k-l+2}$, the elements in the surpluses of $U_1$ and $U_2$ are increasing.
Moreover, the elements in the surpluses of $U_1$ and $U_2$ are all larger than the elements contained in the first row and first column of the corresponding protected areas.
Thus we indeed obtain a pair of tableaux  $(U_1,U_2)$ in $P_{n,k} \times P_{n,k}$ with the same values for $l$ and $m$ as for $(T_1,T_2)$.
For an example of the map $\psi_{n,k}$ with $n=15$, $k=5$ and $(l,m)=(4,12)$, see Figure~\ref{fig:example_protected}. 
The injectivity of the maps $\psi_{n,k}$ clearly follows from the injectivity of the maps $\varphi_{n-m+1,k-l,k-l+2}$.
\end{proof}

\section{Lattice paths and $321$-avoiding permutations}

In the following, let $A_{n,k}$ denote the set of all $321$-avoiding involutions of length $n$  with  longest increasing sequence of length $k$.
The cardinality of $A_{n,k}$ will be denoted by $a_{n,k}$. 
The Robinson-Schensted correspondence maps elements of $A_{n,k}$ into  SYT  that consist of at most two rows such that  the length of the first row is  $k\geq \lceil n/2 \rceil$. Then the hooklength formula (see Chapter 14 in \cite{handbook}) yields that
\[a_{n,k} = {n\choose k} \frac{2k-n+1}{k+1}.\] 
Using this formula, it is routine to prove that for any fixed $n$, the sequence $a_{n,k}$ is log-concave if
$k\in [n/2, n]$, but that proof is not particularly elucidating. In what follows, we provide a more elegant, injective proof.

There is a natural bijection $f=f_n$ between the set of such 
SYT and lattice paths using steps $(0,1)$ and $(1,0)$ that start at $(0,0)$, consist of $n$ steps, and never go above
the diagonal $x=y$. 
We will refer to these steps as East and North steps. Indeed, if $i_1,i_2,\cdots ,i_k$ are the numbers in the first row of the  SYT $T$, and $j_1,j_2,\cdots ,j_{n-k}$ are the numbers in the second row  of $T$,  then we can set $f(T)$ to be the lattice path starting at $(0,0)$ whose East steps are in positions  $i_1,i_2,\cdots ,i_k$, and whose
North steps are in positions $j_1,j_2,\cdots ,j_{n-k}$. The fact that $T$ is a SYT means that $i_t<j_t$ for all $t\leq n-k$,
so $f(T)$ will indeed stay below the diagonal $x=y$ since its $t$-th North step will come some time after its $t$-th East step. 
Note that elements of $A_{n,k}$ will be mapped into paths that end in $(k,n-k)$. 

For an example, see Figure~\ref{fig:pairofpaths}: The path $Q$ that is marked by circles corresponds to the SYT with elements $1,3,4,5,6,7$ in the first row and the element $2$ in the second row.

Let $L(n,k)$ be the set of lattice paths using steps $(0,1)$ and $(1,0)$ that start at $(0,0)$, never go above
the diagonal $x=y$, and end in $(k,n-k)$, where $n-k\leq k\leq n-2$. We define a map \[\phi : L(n,k) \times L(n,k+2)
\rightarrow  L(n,k+1) \times L(n,k+1)\]
as follows. 

Let $(P,Q)\in  L(n,k) \times L(n,k+2)$. Let us translate $P$ by the vector $(1,-1)$ to obtain the path $P'$ that starts
at $(1,-1)$, ends in $(k+1,n-k-1)$, and {\em never goes above the diagonal} $x-2=y$. As $Q$ starts West of $P'$ and ends
East of $P'$, the paths $P'$ and $Q$ will have at least one point in common. Let $X$ be the last such point. Note that
$X$ is not the endpoint of $P'$ or $Q$, since those two paths do not end in the same point. 
We now ``flip'' the paths at $X$, which means the following. Let $P'_1$ and $Q_1$ denote the parts of $P'$ and $Q$ that 
end in $X$, and let $P'_2$ and $Q_2$ denote the parts of $P'$ and $Q$ that start in $X$. Then we define 
$\phi(P,Q)$ as $( (P_1'Q_2)*, Q_1P_2')$, where $(P_1'Q_2)*$  denotes the path $P_1'Q_2$ translated back by the vector
$(-1,1)$ so that it starts at $(0,0)$. See Figure \ref{fig:pairofpaths} for an illustration. 

\begin{figure}
    \centering
    \hspace{0.02\textwidth}
    \begin{minipage}{.45\textwidth}
        \centering
        \begin{tikzpicture}[scale=0.8]

\foreach \x/\y in {0/0, 1/0, 1/1, 2/1, 3/1,4/1,5/1,6/1}
\node[fill=black,circle, inner sep=0.8mm] at (\x,\y) {};

\draw (0,0)--(1,0)--(1,1)--(2,1) -- (3,1)--(4,1)--(5,1)--(6,1); 
   
\foreach \x/\y in {1/-1, 2/-1, 3/-1, 3/0, 4/0,4/1,4/2,5/2}
\node[fill=black,rectangle, inner sep=1mm] at (\x,\y) {};
    
\draw[dashed] (1,-1)--( 2,-1)--(3,-1)--( 3,0)--(4,0)--(4,1)--(4,2)--(5,2);

\node[anchor=north] at (1.5,-1) {$P_1'$};
\node[anchor=south] at (4.5,2) {$P_2'$};
\node[anchor=north] at (0.5,0) {$Q_1$};
\node[anchor=south] at (5.5,1) {$Q_2$};

\node[anchor=south west] at (4,1) {$X$};
    
\draw[dotted] (0,0)--(3,3);
\draw[loosely dotted] (1,-1)--(5,3);
\end{tikzpicture}
    \end{minipage}%
    \hfill
    \begin{minipage}{0.45\textwidth}
        \centering
     \begin{tikzpicture}[scale=0.8]

\foreach \x/\y in {0/0, 1/0, 1/1, 2/1, 3/1,4/1,4/2,5/2}
\node[fill=black,circle, inner sep=0.8mm] at (\x,\y) {};

\draw (0,0)--(1,0)--(1,1)--(2,1) -- (3,1)--(4,1)--(4,2)--(5,2);
   
\foreach \x/\y in {1/-1, 2/-1, 3/-1, 3/0, 4/0,4/1,5/1,6/1}
\node[fill=black,rectangle, inner sep=1mm] at (\x,\y) {};
    
\draw[dashed] (1,-1)--( 2,-1)--(3,-1)--( 3,0)--(4,0)--(4,1)--(5,1)--(6,1); 

\node[anchor=north] at (1.5,-1) {$P_1'$};
\node[anchor=south] at (4.5,2) {$P_2'$};
\node[anchor=north] at (0.5,0) {$Q_1$};
\node[anchor=south] at (5.5,1) {$Q_2$};

\node[anchor=south west] at (4,1) {$X$};

\draw[dotted] (0,0)--(3,3);
\draw[loosely dotted] (1,-1)--(5,3);
\end{tikzpicture}   
    \end{minipage}
    \hspace{0.05\textwidth}
\caption{The main step in the construction of the map $\phi$. }
\label{fig:pairofpaths}
\end{figure}
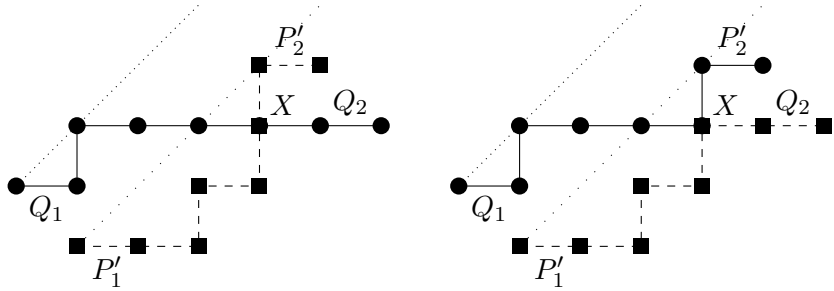

\begin{proposition}
The map $\phi$ described above indeed maps into the set $ L(n,k+1) \times L(n,k+1)$ and is an injection from  $L(n,k) \times L(n,k+2)$ into $L(n,k+1) \times L(n,k+1)$.
\end{proposition}

\begin{proof}
It is a direct consequence of the definitions that both $(P_1'Q_2)*$ and $Q_1P_2'$ will indeed start at $(0,0)$ and end
in $(k+1,n-k+1)$. So, all we need to do in order to  prove that $\phi(P,Q)\in L(n,k+1) \times L(n,k+1)$ is to show that
neither $(P_1'Q_2)*$ and $Q_1P_2'$ ever goes above the diagonal $x=y$. 

\begin{itemize}
\item  To show that $(P_1'Q_2)*$ does not go above the diagonal $x=y$ is equivalent to showing that
$P_1'Q_2$ does not go above the diagonal $x-2=y$. This is true for $P_1'$ by its definition (it is a part of $P'$), and this is true for $Q_2$ since
$Q_2$ is entirely below $P_2'$, and $P_2'$, by its definition (it is a part of $P'$), never goes above the diagonal $x-2=y$.
\item It is clear that $Q_1P_2'$ never goes above the diagonal $x=y$, since neither $Q_1$ (a part of $Q$) nor
$P_2'$ (a part of $P$) do. 
\end{itemize}

Finally, to prove that $\phi$ is injective, let $(R,S)\in  L(n,k+1) \times L(n,k+1)$. 
If $R$ and $S$ have no points 
in common (other than their starting and ending points), $(R,S)$ has no preimage under $\phi$. 
Otherwise, we can recover $X$ as the last point that $R$ and $S$ have in common other than their endpoint, and then reversing
the ``flipping'' operation described in the construction of $\phi$ we can recover the unique preimage of $(R,S)$.  
\end{proof}

\begin{corollary} For any fixed $n$, the sequence $(a_{n,k})_{n/2\leq k\leq n}$ is log-concave. 
\end{corollary}

Therefore, it follows by the principle that we used to prove Theorem \ref{invtoperm} that we have an injective proof of the following corollary as well.

\begin{corollary}
Let $b_{n,k}$ denote the number of 321-avoiding permutations of length $n$ in which the longest increasing subsequence is of length
$k$. Then for any fixed $n$, the sequence $(b_{n,k})_{n/2\leq k\leq n}$ is log-concave. 
\end{corollary}

\section{Concluding remarks}

In this paper we characterized several sets of permutations for which our Conjecture~\ref{conj} holds.
One main tool was to first prove our results for involutions and then to transfer them to arbitrary permutations.
A next step in this line of research and towards proving our conjecture in general would be to find larger sets of permutations for which this method can be applied.
Also, it would be interesting to find applications of this technique to other permutation statistics than the length of the longest increasing subsequence.

\bibliographystyle{abbrv}
\bibliography{lit}

\end{document}